\documentclass[11pt,a4paper]{article}

\usepackage[pdftex]{graphicx}
\usepackage{amssymb, amsmath, amsfonts}
\usepackage{amsthm}
   \theoremstyle{definition}
    \newtheorem{dfn}{Definition}[section]
    \newtheorem{prop}[dfn]{Proposition}
    \newtheorem{lem}[dfn]{Lemma}
    \newtheorem{thm}[dfn]{Theorem}
    \newtheorem{cor}[dfn]{Corollary}
    \newtheorem{rem}[dfn]{Remark}
    \newtheorem{examp}[dfn]{Example}

\makeatletter
    
    \@addtoreset{equation}{section}
  \makeatother
\usepackage{cases} 
\usepackage{mathtools}
\usepackage{here}
\numberwithin{equation}{section}

\usepackage[all]{xy}

\usepackage{color, hyperref}
\usepackage{cite}

\usepackage[dvipsnames]{xcolor}
\hypersetup{
	colorlinks=true,
	linkcolor=Magenta,
    citecolor=Cyan,
}

\title{Orientation double covers of non-orientable Lefschetz fibrations}
\author{Tomoya Yoshikawa}
\date{}

\begin{document}

\maketitle

\begin{abstract}
In this paper, we prove that the composition of the standard orientation double covering map and a non-orientable Lefschetz fibration is an achiral Lefschetz fibration and specify a monodromy factorization of this composition. As an application of these results, we also give three transformations with respect to monodromy factorizations of non-orientable Lefschetz fibrations which do not change their isomorphism classes via the similar result to orientable Lefschetz fibrations.
\end{abstract}

\section{Introduction.}

Lefschetz fibrations are widely recognized as a tool for studying orientable 4-manifolds. One reason for this is that they serve as a bridge connecting symplectic topology and mapping class groups of orientable surfaces. Donaldson \cite{Donaldson} showed that every closed symplectic 4-manifold admits a Lefschetz fibration over a sphere after some blow-ups, and Gompf \cite{G and S} showed that the total space of a Lefschetz fibration has a symplectic structure if its regular fibers are not trivial in homology. Given a Lefschetz fibration over a sphere, one can obtain a product of Dehn twists satisfying some conditions, which is called its monodromy factorization. Conversely, given such a product of Dehn twists, a Lefschetz fibration over a sphere is constructed (see \cite{G and S} for example).\par
Another reason why Lefschetz fibrations are attractive for the study of orientable 4-manifolds is that one can construct exotic orientable 4-manifolds via mapping class groups of orientable surfaces. Kas \cite{Kas} and Matsumoto \cite{Matsumoto} connected isomorphism classes of Lefschetz fibrations to Hurwitz equivalence for their monodromy factorizations using monodromies (see Theorem \ref{thm4.5} in this paper). Moreover, it became clear, by \cite{E and G} and \cite{E and M and V}, that lantern and daisy substitutions correspond to rational blowdowns (cf. \cite{F and S}) on orientable 4-manifolds. These results have contributed to combinatorial constructions of exotic orientable 4-manifolds.\par
Recently, attempts to understand non-orientable 4-manifolds via Lefschetz fibrations are made (see \cite{M and O} for example). Research on this topic is still insufficient compared to the orientable case, but as can be inferred from the above, it is expected to contribute to combinatorial discussions of non-orientable 4-manifolds.\par
In this paper, we study non-orientable Lefschetz fibrations focusing on the relationship between them and orientable ones. In order to do this, we first show the following proposition in subsection \ref{subsec.2.3}, which associates orientable and non-orientable Lefschetz fibrations.

\begin{prop}\label{prop1.1}
\textit{Let \(f:X \to \Sigma\) be a genus-g non-orientable Lefshetz fibration and \(\pi:\widetilde{X} \to X\) the standard orientation double covering map. Then the composition \(f \circ \pi:\widetilde{X} \to \Sigma\) is an achiral Lefschetz fibration of genus-\((g-1)\) with the same number of positive critical points and negative critical points.}
\end{prop}

\noindent
Here, we will prepare the standard orientation double covering map in Example \ref{examp2.3}.\par
As mentioned above, it was shown in \cite{Kas} and \cite{Matsumoto} that two orientable Lefschtz fibrations are isomorphic if and only if their monodromy factorizations are Hurwitz equivalent. On the other hand, we prove the non-orientable analogue of this result via orientable Lefschetz fibrations, and this statement is claimed as follows, which is our main theorem in this paper.

\begin{thm}\label{thm1.2}
\textit{
For \(g \geqq 3\), two non-orientable Lefschetz fibrations of genus-\(g\) over either \(\mathbb{D}^2\) or \(\mathbb{S}^2\) are isomorphic if and only if their positive monodromy factorizations change each other by the following transformations:}
\begin{itemize}
\item[(i-1)]
\begin{gather*}
{t_{c_1;\theta_{c_1}}} \cdots {t_{c_i;\theta_{c_i}}} {t_{c_{i+1};\theta_{c_{i+1}}}} {t_{c_{i+2};\theta_{c_{i+2}}}} \cdots {t_{c_n;\theta_{c_{n}}}} \\
\leftrightarrow {t_{c_1;\theta_{c_1}}} \cdots {t_{c_i;\theta_{c_i}}} {t_{c_{i+1};\theta_{c_{i+1}}}} ({t_{c_i;\theta_{c_i}}}^{-1} {t_{c_i;\theta_{c_i}}}) {t_{c_{i+2};\theta_{c_{i+2}}}} \cdots {t_{c_n;\theta_{c_{n}}}};
\end{gather*}
\item[(i-2)]
\begin{gather*}
{t_{c_1;\theta_{c_1}}} \cdots {t_{c_{i-1};\theta_{c_{i-1}}}} {t_{c_{i};\theta_{c_{i}}}} {t_{c_{i+1};\theta_{c_{i+1}}}} \cdots {t_{c_n;\theta_{c_{n}}}} \\
\leftrightarrow {t_{c_1;\theta_{c_1}}} \cdots {t_{c_{i-1};\theta_{c_{i-1}}}} ({t_{c_{i+1};\theta_{c_{i+1}}}} {t_{c_{i+1};\theta_{c_{i+1}}}}^{-1}) {t_{c_{i};\theta_{c_{i}}}} {t_{c_{i+1};\theta_{c_{i+1}}}} \cdots {t_{c_n;\theta_{c_{n}}}};
\end{gather*}
\item[(ii)]
\[
{t_{c_1;\theta_{c_1}}} \cdots {t_{c_n;\theta_{c_n}}} \leftrightarrow {t_{\omega(c_1);\omega_*(\theta_{c_1})}} \cdots {t_{\omega(c_n);\omega_*(\theta_{c_n})}},
\]
\textit{for an auto-diffeomorphism} \(\omega:N_g \to N_g\) of the non-orientable, closed, connected, genus-\(g\) smooth surface \(N_g\).
\end{itemize}
\end{thm}

\noindent
Here, for a two-sided simple closed curve \(c\) on \(N_g\), \({t_{c;\theta_{c}}}\) denotes the right handed Dehn twist about \(c\) considering an orientation \(\theta_c\) of a tubular neighborhood of \(c\). Further details are provided in Subsection \ref{subsec.4.1}. \par
The organization of this article is as follows. In Section \ref{sec.2}, we recall the definition of orientation double covering map (Definition \ref{def2.6}) and their some properties (Lemma \ref{lem2.4} and \ref{lem2.5}). Furthermore, we also review the definition of Lefschetz fibrations (Defintion \ref{def2.6}), and then we show Proposition \ref{prop1.1}. In Section \ref{sec.3}, we identify vanishing cycles of the achiral Lefschetz fibration in Proposition \ref{prop1.1} (Proposition \ref{prop3.1}). In Section \ref{sec.4}, we give the the proof of Theorem \ref{thm1.2}. For this, we introduce an injective homomorphism between the mapping class group of the genus-\(g\) non-orientable surface and the mapping class group of the genus-\((g-1)\) orientable one. This homomorphism associates the monodromy representation of non-orientable Lefschetz fibrations with that of orientable ones (Proposition \ref{prop4.2}), and we use this to lift a monodromy factorization of a non-orientable Lefschetz fibration to that of an orientable one (Corollary \ref{cor4.3}). Applying these results, we finally prove Theorem \ref{thm1.2}.\par
\noindent
\textbf{Acknowledgements:}
The author would like to express his gratitude to Professor Naoyuki Monden for meaningful discussion with him and his much encouragement. He would also like to thank Professors Masatoshi Sato and Kenta Hayano since Proposition \ref{prop1.1} in this paper is inspired by their useful comments.\par
\noindent
\textbf{Notations:}
Throughout this paper, \(\Sigma_g\) (resp. \(N_g\)) denotes the orientable (resp. non-orientable), closed, connected, genus-\(g\) smooth surface. We write the unit disk and the unit sphere as \(\mathbb{D}^2\) and \(\mathbb{S}^2\), respectively. Moreover, let \(\Sigma_g^b\) denote the orientable, compact, connected, genus-\(g\) smooth surface with \(b\) boundary components. We will regard \(N_{2g+1}\) as the surface which is obtained by removing an open disk from \(\Sigma_g\) and gluing a M\(\mathrm{\ddot{o}}\)bius band along the boundary of the resulting surface. Similarly, we will regard \(N_{2g+2}\) as the surface which is obtained by removing two open disks from \(\Sigma_g\) and gluing two M\(\mathrm{\ddot{o}}\)bius bands along the boundaries of the resulting surface. Then each of these M\(\mathrm{\ddot{o}}\)bius bands is drawn on \(N_g\) as \(\otimes\) and is said to be a \textit{cross cap} (see Figure \ref{Fig.1}). 
A simple closed curve on \(N_g\) is called \textit{two-sided} if its tubular neighborhood is not diffeomorphic to a M\(\mathrm{\ddot{o}}\)bius band but diffeomorphic to an annulus. A two-sided simple closed curve on \(N_g\) is drawn as a curve which passes through some cross caps an even number of times (see Figure \ref{Fig.1}).   

\begin{figure}[H]
\centering
\includegraphics[width=100mm]{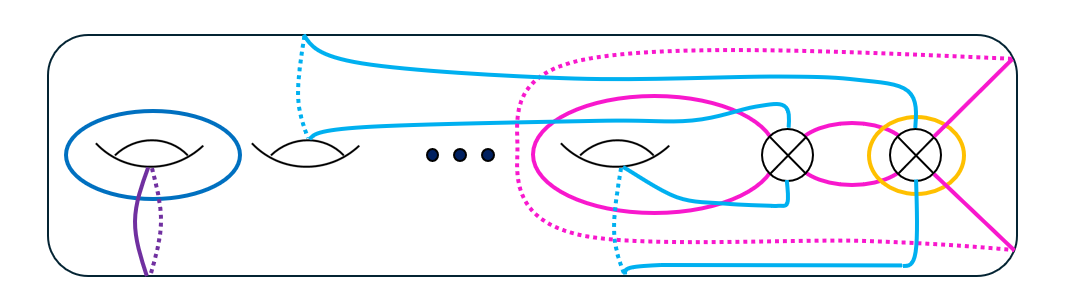}
\caption{Cross caps and two-sided simple closed curves on \(N_{2g+2}\).}
\label{Fig.1}
\end{figure}

\section{Orientation double covering maps and Lefschetz fibrations.}\label{sec.2}

\subsection{Orientation double covering maps.}\label{subsec.2.1}

In this subsection, we recall some properties of orientation double covering maps. More details can be found in \cite{Lima} for example. 

\begin{dfn}\label{def2.1}
Let \(M\) and \(N\) be smooth manifolds. A smooth map \(\Theta:M \to N\) is said to be \textit{orientation double covering} if the following conditions are satisfied:
\begin{itemize}
\item[(1)] 
\(N\) is connected, \(M\) is orientable, and \(\Theta\) is a local diffeomorphism;
\item[(2)]
For each \(q \in N\),the inverse image \(\Theta^{-1}(q)\) is a finite set consisting of two points;
\item[(3)]
Fix \(q \in N\), and set \( \{p_1,p_2\}:=\Theta^{-1}(q) \). Then the linear isomorphism
\begin{gather*}
{d(\Theta|_{(\Theta|_{U_{p_2}})^{-1}(\Theta(U_{p_1}) \cap \Theta({U_{p_2}}))})_{p_2}}^{-1} \circ d(\Theta|_{(\Theta|_{U_{p_1}})^{-1}(\Theta(U_{p_1}) \cap \Theta({U_{p_2}}))})_{p_1} \\
:T_{p_1}M \to T_{p_{2}}M
\end{gather*}
is orientation reversing, where for \(i=1,2\), \(U_{p_i}\) is an open neighborhood of \(p_i\) giving a local diffeomorphism at \(p_i\).
\end{itemize}
\end{dfn}

\begin{examp}\label{examp2.2}
Consider the map \(J:\Sigma_{g-1} \to \Sigma_{g-1}\) defined by mapping each point to its antipodal point. Since the relation \(p \sim J(p)\) generates an equivalence reltion on \(\Sigma_{g-1}\), we obtain the quotient map \(\widetilde{J}:\Sigma_{g-1} \to \Sigma_{g-1}/\sim \thickspace =N_g\). \(\widetilde{J}\) is a double covering map, especially an orientation double covering map.
\end{examp}

\begin{examp}\label{examp2.3}
Let \(M\) be an \(m\)-dimensional smooth manifold with a smooth atlas \(\{(U_{\lambda},\varphi_{\lambda})\}_{\lambda \in \Lambda}\) and \(\mathrm{ob}(T_pM)\) the set of all ordered bases for the tangent space \(T_pM\) at \(p \in M\). We say that two ordered bases are equivalent if the change-of-basis matrix has positive determinant. This gives an equivalence relation ‘\(\sim\)’ on \(\mathrm{ob}(T_pM)\), and hence, the quotient set \(\mathrm{ob}(T_pM)/\sim\) of \(\mathrm{ob}(T_pM)\) is defined. The set \(\mathrm{ob}(T_pM)/\sim\) consists of exactly two elements, which are called local orientations of \(M\) around \(p\). Now, let \(\widetilde{M}\) be the set
\[
\widetilde{M}:=\{(p,o_{p}) \thickspace | \thickspace p \in M, o_{p} \in (\mathrm{ob}(T_pM)/\sim) \},
\]
and define a map \(\pi:\widetilde{M} \to M\) by
\[
\pi(p,o_p):=p
\]
for any \((p,o_p) \in \widetilde{M}\). Moreover, for each \(\lambda \in \Lambda\), we define subsets \(U_{\lambda}^+, U_{\lambda}^-\) of \(\widetilde{M}\) and a map \({\varphi_{\lambda}}^\dag:{U_{\lambda}}^{\dag} \to \mathbb{R}^m \thickspace (\dag \in \{+,-\})\) as follows:
\[
U_{\lambda}^+:=\Big\{(p,o_p) \in \widetilde{M} \thickspace | \thickspace p \in U_{\lambda}, o_p=\Big[\Big( \frac{\partial}{\partial x_1^\lambda}\Big|_p, \cdots , \frac{\partial}{\partial x_m^\lambda}\Big|_p \Big) \Big] \Big\},
\]
\[
U_{\lambda}^-:=\Big\{(p,o_p) \in \widetilde{M} \thickspace | \thickspace p \in U_{\lambda}, o_p=-\Big[\Big( \frac{\partial}{\partial x_1^\lambda}\Big|_p, \cdots , \frac{\partial}{\partial x_m^\lambda}\Big|_p \Big) \Big] \Big\},
\]
\[
{\varphi_{\lambda}}^\dag := \varphi_{\lambda} \circ \pi|_{U_{\lambda}^{\dag}},
\]
\noindent
where \(x_1^\lambda, \cdots ,x_m^\lambda\) is the local coordinate on \(U_{\lambda}\), and \(\Big[\Big( \frac{\partial}{\partial x_1^\lambda}\Big|_p, \cdots , \frac{\partial}{\partial x_m^\lambda}\Big|_p \Big) \Big] \in \mathrm{ob}(T_pM)/\sim\) is the equivalence class of the ordered basis \(\Big( \frac{\partial}{\partial x_1^\lambda}\Big|_p, \cdots , \frac{\partial}{\partial x_m^\lambda}\Big|_p \Big)\) for \(T_pM\). Then we see that \(\pi\) is a double covering map, especially an orientation double covering map and \(\{({U_{\lambda}}^{+},{\varphi_{\lambda}}^{+}),({U_{\lambda}}^{-},{\varphi_{\lambda}}^{-})\}_{\lambda \in \Lambda}\) provides a smooth atlas of \(\widetilde{M}\). From now on, to simplify notation, let \(+o_p\) (resp. \(-o_p\)) denote \(\Big[\Big( \frac{\partial}{\partial x_1^\lambda}\Big|_p, \cdots , \frac{\partial}{\partial x_m^\lambda}\Big|_p \Big) \Big]\) (resp. \(-\Big[\Big( \frac{\partial}{\partial x_1^\lambda}\Big|_p, \cdots , \frac{\partial}{\partial x_m^\lambda}\Big|_p \Big) \Big]\)). We remark that orientations of \(T_{(p,+o_p)} \widetilde{M}\) and \(T_{(p,-o_p)} \widetilde{M}\) are exactly specified by 
\[\Big[\Big( (d\pi_{(p,+o_p)})^{-1}\Big(\frac{\partial}{\partial x_1^\lambda}\Big|_p\Big), \cdots, (d\pi_{(p,+o_p)})^{-1}\Big(\frac{\partial}{\partial x_m^\lambda}\Big|_p\Big) \Big) \Big]\]
and
\[-\Big[\Big( (d\pi_{(p,-o_p)})^{-1}\Big(\frac{\partial}{\partial x_1^\lambda}\Big|_p\Big), \cdots , (d\pi_{(p,-o_p)})^{-1}\Big(\frac{\partial}{\partial x_m^\lambda}\Big|_p\Big) \Big) \Big],\]
which implies that \(\varphi_{\lambda}^+\) and \(\varphi_{\lambda}^-\) are orientation preserving and orientation reversing, respectively. Throughout this paper, we will call the orientation double covering map \(\pi\) \textit{standard}. 
\end{examp}

Orientation double covering maps have the following properties.

\begin{lem}\label{lem2.4}
\begin{itemize}
\item[(1)] 
\textit{Every connected manifold has an orientation double covering map.}
\item[(2)] 
\textit{Let \(\Theta:M \to N\) be an orientation double covering map. Then \(M\) is connected if and only if \(N\) is non-orientable.}
\end{itemize}
\end{lem}

\noindent
In particular, for various discussions in this paper, we need the next lemma. 

\begin{lem}\label{lem2.5}
\textit{
For \(i=1,2\), let \(\Theta_i:M_i \to N\) be an orientation double covering map. Then there is a unique orientation preserving diffeomorphism \(\kappa:M_1 \to M_2\) such that \(\Theta_1 = \kappa \circ \Theta_2\).
}
\end{lem}

It is well-known that the covering transformation group of an orientation double covering map \(\Theta:M \to N\) denoted by \(\mathrm{Aut}(\Theta)\) consists of exactly two maps, which are the identity map and the orientation reversing involution.

\subsection{Lefschtez fibrations.}
From now on, for a smooth manifold \(Q\), \(\mathrm{Int}(Q)\) and \(\partial Q\) denote the interior and the boundary of \(Q\), respectively. In addition, for a smooth map \(F\) between two manifolds, we write the set of all critical points of \(F\) as \(\mathrm{Crit}(F)\) and the image \(F(\mathrm{Crit}(F))\) as \(\Delta_F\). 

\begin{dfn}\label{def2.6}
Let \(X\) be a compact, connected, smooth 4-manifold and \(\Sigma\) a compact, connected, orientable smooth 2-manifold. A smooth map \(f:X \to \Sigma\) is called a \textit{Lefschetz fibration} if it satisfies the following conditions: 
\begin{itemize}
\item[(LF1)] 
\(\partial X=f^{-1}(\partial \Sigma)\);
\item[(LF2)]
\(\mathrm{Crit}(f) \subset \mathrm{Int}(X)\);
\item[(LF3)]
for each \(x \in \mathrm{Crit}(f)\), there are complex coordinate charts \((U,\varphi)\) and \((V,\psi)\) about \(x\) and \(f(x)\), respectively, such that 
\[
(\psi \circ f \circ \varphi^{-1})(z,w)=zw
\]
for all \((z,w) \in \varphi(U) \subset \mathbb{C}^2\).
\end{itemize}
We call \(X\) and \(\Sigma\) the \textit{total space} and the \textit{based space} of a Lefschetz fibration \(f:X \to \Sigma\), respectively.
\end{dfn}

If the total space \(X\) of a Lefschetz fibration \(f:X \to \Sigma\) is orientable, then \(f\) is said to be \textit{orientable}. Furthermore, an orientable Lefschetz fibration is called \textit{achiral} if there exists a chart \((U,\varphi)\) about a critical point of \(f\) satisfying the condition (LF3) such that its local coordinate system \(\varphi\) is orientation reversing. We call a critical point \(x \in \mathrm{Crit}(f)\) \textit{positive} (resp. \textit{negative}) if the local coordinate system of a chart about it with the condition (LF3) is orientation preserving (resp. orientation reversing). For each \(y \in \Sigma - \Delta_f\), the inverse image \(f^{-1}(y)\) of \(y\), which is called a \textit{regular fiber}, is diffeomorphic to \(\Sigma_g \thickspace (g \geqq 0\)), and we then say that \(f\) has the genus \(g\). A \textit{singular fiber} of \(f\) is the inverse image of a critical value of \(f\). 
\par
On the other hand, if the total space \(X\) of a Lefschetz fibration \(f:X \to \Sigma\) is non-orientable, then \(f\) is said to be \textit{non-orientable}. In this case, a regular fiber of \(f\) is diffeomorphic to \(N_g \thickspace (g \geqq 1\)).

\begin{rem}\label{rem2.7}
Let \(f:X \to \Sigma\) be a Lefschetz fibration. Since the condition (LF3) and compactness of X, it follows that \(\mathrm{Crit}(f)\) is a finite set. Moreover, we can assume that \(f|_{\mathrm{Crit}(f)}\) is an injection, i.e., each singular fiber of \(f\) has only one critical point by perturbing \(f\) slightly.
\end{rem}

\subsection{Proof of Proposition 1.1.}\label{subsec.2.3}

Let \(f:X \to \Sigma\) be a genus-\(g\) non-orientable Lefshetz fibration and \(\pi:\widetilde{X} \to X\) the standard orientation double covering map. We now prove Proposition \ref{prop1.1}.

\begin{proof}[Proof of Proposition \ref{prop1.1}.]
We first show that \(f \circ \pi\) satisfies the conditions (LF1) and (LF2). Since \(\pi\) is a local diffeomorphism, we observe, by the excision theorem, that \(H_4(\widetilde{X},\widetilde{X}-\{(x,o_x)\}) \cong H_4(X,X-\{x\})\)  for any \(x \in X\). Thus we have
\begin{equation*}
    \begin{split} 
\partial \widetilde{X} 
&= \{(x,o_x) \in \widetilde{X} \thickspace | \thickspace H_4(\widetilde{X},\widetilde{X}-\{(x,o_x)\})=0\} \\
&= \{(x,o_x) \in \widetilde{X} \thickspace | \thickspace H_4(X,X-\{x\})=0\} \\
&= \pi^{-1}(\partial X) \\
&= (f \circ \pi)^{-1}(\partial \Sigma).
    \end{split}
\end{equation*}
Since \(\pi\) is a smooth covering map, it is in particular a submersion. Hence, we see that
\begin{equation*}
    \begin{split} 
\mathrm{Crit}(f \circ \pi)
&=\{(x,o_x)\in \widetilde{X} \thickspace | \thickspace x \in \mathrm{Crit}(f)\} \\
&=\pi^{-1}(\mathrm{Crit}(f)),
    \end{split}
\end{equation*}
and this gives \(\mathrm{Crit}(f \circ \pi) \subset \mathrm{Int}(\widetilde{X})\).
\par
Next, we show that \(f \circ \pi\) satisfies the condition (LF3). Fix \(x \in \mathrm{Crit}(f)\). Take charts \((U,\varphi)\) and \((V,\phi)\) about \(x\) and \(f(x)\), respectively, such that they satisfy the condition (LF3). From the discussion of Example \ref{examp2.3}, for \((U,\varphi)\), \((U^\dag,\varphi^\dag)\) is a chart about \((x,\dag o_x) \in \widetilde{X}\) for each \(\dag \in \{+,-\}\), and we then get 
\begin{equation*}
    \begin{split} 
(\psi \circ (f \circ \pi) \circ (\varphi^\dag)^{-1})(z,w) 
&=(\psi \circ f \circ \varphi^{-1})(z,w) \\
&=zw
    \end{split}
\end{equation*}
for all \((z,w) \in \varphi^\dag(U^\dag)=\varphi(U) \subset \mathbb{C}^2\). Therefore, combining this with the fact that \(\varphi^+\) and \(\varphi^-\) are orientation preserving and orientation reversing, respectively, it follows that \(f \circ \pi\) is an achiral Lefschetz fibration which has the same number of positive critical points and negative critical points.
\par
Finally, we show that the genus of \(f \circ \pi\) equals to \(g-1\). Fix \(y \in \Sigma - \Delta_f\) and a diffeomorphism \(\Xi:f^{-1}(y) \to N_g\). Note that the composition \(\Xi \circ \pi|_{(f \circ \pi)^{-1}(y)}:(f \circ \pi)^{-1}(y) \to N_g\) is an orientation double covering map. Hence, applying Lemma \ref{lem2.5}, we can obtain a unique orientation preserving diffeomorophism \(\widetilde{\Xi}:(f \circ \pi)^{-1}(y) \to \Sigma_{g-1}\) such that the following diagram commutes:

\begin{equation}\label{diagram(2.1)}
\xymatrix{
(f \circ \pi)^{-1}(y)\ar[r]^-{\widetilde{\Xi}}\ar[d]_-{\pi|_{(f \circ \pi)^{-1}(y)}}\ar@{}[rd]|{ }&\Sigma_{g-1}\ar[d]^-{\widetilde{J}}\\
f^{-1}(y) \ar[r]_-{\Xi}&N_g,
}
\end{equation}

\noindent
where the map \(\widetilde{J}\) is defined in Example \ref{examp2.2}. This yields our desired result.
\par
We complete the proof of Proposition \ref{prop1.1}.
\end{proof}

\begin{rem}
In Proposition \ref{prop1.1}, an alternative proof of the genus of \(f \circ \pi:\widetilde{X} \to \Sigma:=\Sigma_{g'}^b\) is given using the Euler characteristic of a Lefschetz fibration, under the assumption that \(2g'+b \ne 2\). Suppose that \(f\) has \(n\) singular fibers. Then 
\begin{equation}\label{(2.2)}
\chi(X)=\chi(N_g)\chi(\Sigma)+n
\end{equation}
holds, where \(\chi\) indicates the Euler characteristic of a topological space. Note here that
\begin{equation}\label{(2.3)}
\chi(\widetilde{X})=2\chi(X)
\end{equation}
because \(X\) is a smooth manifold, i.e., a CW complex and the degree of the covering map \(\pi:\widetilde{X} \to X\) is two.  Moreover, according to the proof of Proposition \ref{prop1.1} and Remark \ref{rem2.7}, the number of singular fibers of \(f \circ \pi\) is \(2n\). So we have
\begin{equation}\label{(2.4)}
\chi(\widetilde{X})=\chi(\Sigma_k)\chi(\Sigma)+2n,
\end{equation}
here we assume that the genus of \(f \circ \pi\) is \(k\). 
It follows from the equations \eqref{(2.2)}, \eqref{(2.3)}, and \eqref{(2.4)} that \(k=g-1\), and we thus get the assertion.

\end{rem}

\section{Vanishing cycles of Lefschetz fibrations.}\label{sec.3}

The aim of this section is to identify vanishing cycles of the achiral Lefschetz fibration which is constructed by the composition of the standard orientation double covering map and a non-orientable Lefschetz fibration. In order to do this, we first review the construction of vanishing cycles of a Lefschetz fibration.
\par
Define a map \(\bar{f}:\mathbb{C}^2 \to \mathbb{C}\) by \(\bar{f}(z,w):=zw\) for all \((z,w) \in \mathbb{C}^2\). It is easy to see that the set
\begin{equation}\label{(3.1)}
\bar{f}^{-1}(\varepsilon) \cap \{(z,w) \in \mathbb{C}^2 \thickspace | \thickspace |z|=|w|\}
\end{equation}
is diffeomorphic to the circle with center \(0 \in \mathbb{C}\) and radius \(\sqrt{|\varepsilon|}\) if  \(\varepsilon \ne 0\) and equals to the set consisting of the single point \(0 \in \mathbb{C}^2\) if \(\varepsilon=0\). Let \(f:X \to \Sigma\) be a Lefschetz fibration. Fix \(y_0 \in \Delta_f\) and \(y_1 \in \Sigma - \Delta_f\), and let \(x_0 \in \mathrm{Crit}(f) \cap f^{-1}(y_0)\). We can take a smooth embedding \(\ell:\lbrack 0,1 \rbrack \to \Sigma\) such that \(y_0=\ell(0)\) and \(y_t:=\ell(t) \in \Sigma-\Delta_f\) \((t \in (0,1\rbrack)\). In addition, choose complex coordinate charts \((U,\varphi)\) and \((V,\psi)\) about \(x_0\) and \(y_0\), respectively, such that they satisfy the condition (LF3). Now, for a sufficiently small fixed \(\delta \in (0,1 \rbrack\), we define a simple closed curve as
\[
C_{\psi(y_\delta)}:=(\psi \circ f \circ \varphi^{-1})^{-1}(\psi(y_\delta)) \cap \{(z,w) \in \mathbb{C}^2 \thickspace | \thickspace |z|=|w|\}
\]
based on the set \eqref{(3.1)}, and let \(C\) denote the image of \(C_{\psi(y_\delta)}\) under the diffeomorphism \(\varphi^{-1}:\varphi(U) \to U\). Take a vector field \(W:f^{-1}(\ell(\lbrack \delta,1 \rbrack)) \to T(f^{-1}(\ell(\lbrack \delta,1 \rbrack)))\) on the submanifold \(f^{-1}(\ell(\lbrack \delta,1 \rbrack))\) of \(X\) which satisfies
\[
d(f|_{f^{-1}(\ell(\lbrack \delta,1 \rbrack))})_x(W(x))=\dfrac{d\ell|_{\lbrack \delta,1 \rbrack}}{dt}(t)
\]
for any \(x \in f^{-1}(y_t)\), where \(T(f^{-1}(\ell(\lbrack \delta,1 \rbrack)))\) represents the tangent bundle of \(f^{-1}(\ell(\lbrack \delta,1 \rbrack))\). For each \(x \in f^{-1}(y_\delta)\), we then have an integral curve \(c_x\) of \(W\) starting at \(x\), and therefore, we get the diffeomorphism \(\Upsilon:f^{-1}(y_\delta) \to f^{-1}(y_1)\) given by \(\Upsilon(x):=c_x(1-\delta)\) for all \(x \in f^{-1}(y_\delta)\). Furthermore, we obtain, by using the diffeomorphism \(\Upsilon\), the simple closed curve \(\Upsilon(C) \subset f^{-1}(y_1)\). Note that the isotopy class of \(\Upsilon(C)\) is uniquely determined by the homotopy class of \(\ell\) with fixed endpoints. In the context of the discussion above, the isotopy class of \(\Upsilon(C)\) is said to be the \textit{vanishing cycle} of \(f\) determined by \(\ell\).

\begin{prop}\label{prop3.1}
\textit{
Let \(f:X \to \Sigma\) be a genus-\(g\) non-orientable Lefschetz fibration and \(\pi:\widetilde{X} \to X\) the standard orientation double covering map. Then each vanishing cycle of the genus-\((g-1)\) achiral Lefschetz fibration \(f \circ \pi\) appears as a pair of inverse images of a vanishing cycle of \(f\) under \(\pi\). In particular, each pair is drawn as being an antipodal position on \(\Sigma_{g-1}\) as shown in Figure \ref{Fig.2}.
}  
\end{prop}

\begin{figure}[H]
\centering
\includegraphics[width=110mm]{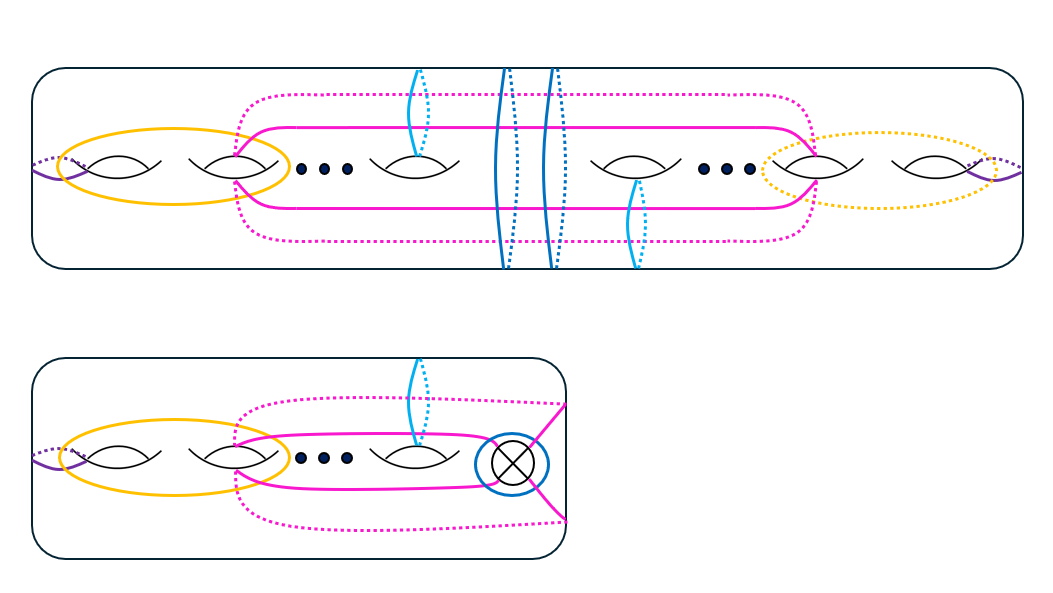}
\caption{Vanishimg cycles of a genus-\(g\) non-orientable Lefschetz fibration \(f\) and the genus-(\(g-1\)) achiral Lefschetz fibration \(f \circ \pi\), where \(g\) is odd.}
\label{Fig.2}
\end{figure}

\begin{proof} 
Let \(f:X \to \Sigma\) be a genus-\(g\) non-orientable Lefschetz fibration, and for this, we will use the same symbols as in the above discussion (for instance, \(x_0,\ell,C,\Upsilon\), etc). It follows from Proposition \ref{prop1.1} and its proof that \(f \circ \pi:\widetilde{X} \to \Sigma\) is an achiral Lefschetz fibration of genus-\((g-1)\) and a chart about \((x_0,\dag o_{x_0}) \in \mathrm{Crit}(f \circ \pi) \cap (f \circ \pi)^{-1}(y_0)\) with the condition (LF3) is given by \((U^{\dag},\varphi^{\dag})\) for each \(\dag \in \{+,-\}\). As above, for each \(\dag \in \{+,-\}\), we can construct the simple closed curve
\[
{C_{\psi(y_\delta)}}^\dag:=(\psi \circ (f \circ \pi) \circ (\varphi^\dag)^{-1})^{-1}(\psi(y_\delta)) \cap \{(z,w) \in \mathbb{C}^2 \thickspace | \thickspace |z|=|w|\}
\]
and put \(C^\dag :=(\varphi^\dag)^{-1}({C_{\psi(y_\delta)}}^\dag)\). Notice that 
\[
C^+ \amalg C^-=\pi^{-1}(C)
\]
holds. Define a map \(\widetilde{W}:(f \circ \pi)^{-1}(\ell(\lbrack \delta,1 \rbrack)) \to T((f \circ \pi)^{-1}(\ell(\lbrack \delta,1 \rbrack)))\) by
\[
d(\pi|_{(f \circ \pi)^{-1}(\ell(\lbrack \delta,1 \rbrack))})_{(x,o_{x})}(\widetilde{W}(x,o_x))=W(x)
\]
for all \((x,o_x) \in (f \circ \pi)^{-1}(y_t)\). Then \(\widetilde{W}\) is a vector field on the submanifold \((f \circ \pi)^{-1}(\ell(\lbrack \delta,1 \rbrack))\) of \(\widetilde{X}\). In addition, for each \((x,o_x) \in (f \circ \pi)^{-1}(y_t)\), the following equality holds:
\[
d(f \circ \pi|_{(f \circ \pi)^{-1}(\ell(\lbrack \delta,1 \rbrack))})_{(x,o_{x})}(\widetilde{W}(x,o_x))=\dfrac{d\ell|_{\lbrack \delta,1 \rbrack}}{dt}(t).
\]
We can define a diffeomorphism \(\widetilde{\Upsilon}:(f \circ \pi)^{-1}(y_{\delta}) \to (f \circ \pi)^{-1}(y_1)\) using an integral curve \(\widetilde{c}_{(x,o_x)}\) of \(\widetilde{W}\) starting at \((x,o_x) \in (f \circ \pi)^{-1}(y_{\delta})\) to be \(\widetilde{\Upsilon}((x,o_x)):=\widetilde{c}_{(x,o_x)}(1-\delta)\) for all \((x,o_x) \in (f \circ \pi)^{-1}(y_{\delta})\). Note that
\[
\pi|_{(f \circ \pi)^{-1}(\ell(\lbrack \delta,1 \rbrack))} \circ \widetilde{c}_{(x,o_x)}=c_x
\]
by the uniqueness of integral curves, and hence, we see that the vanishing cycles \(\widetilde{\Upsilon}(C^+)\) and \(\widetilde{\Upsilon}(C^-)\) of \(f \circ \pi\) satisfy

\begin{equation}\label{(3.2)}
\widetilde{\Upsilon}(C^+) \amalg \widetilde{\Upsilon}(C^-)=\pi^{-1}(\Upsilon(C)). 
\end{equation}

\noindent
Moreover, as mentioned in the proof of Proposition \ref{prop1.1}, for a fixed diffeomorphism \(\Xi:f^{-1}(y_1) \to N_g\), there is exactly one orientation preserving diffomorphism \(\widetilde{\Xi}:(f \circ \pi)^{-1}(y_1) \to \Sigma_{g-1}\) such that \(\Xi \circ \pi|_{(f \circ \pi)^{-1}(y_1)}=\widetilde{J} \circ \widetilde{\Xi}\). It is easy to check that

\begin{equation}\label{(3.3)}
\widetilde{\Xi}(\widetilde{\Upsilon}(C^+) \amalg \widetilde{\Upsilon}(C^-))=\widetilde{J}^{-1}(\Xi(\Upsilon(C))).  
\end{equation}

\noindent
The equations \eqref{(3.2)} and \eqref{(3.3)} mean the desired result.
\end{proof}

\begin{rem}\label{rem3.2}
In the process of proving Proposition \ref{prop3.1}, we also find vanishing cycles of a non-orientable Lefschetz fibration are two-sided. Indeed, since \(\pi|_{U^{\dag}}:U^{\dag} \to U\) is a diffeomorphism for \(\dag \in \{+,-\}\), it follows that \(C=\pi|_{U^{\dag}}(C^{\dag})\) is two-sided, and thus, \(\Upsilon(C)\) is also two-sided.
\end{rem}

\section{Isomorphism classes of non-orientable Lefschetz fibrations and Hurwitz equivalence.}\label{sec.4}

In this section, working from results about orientable Lefschetz fibrations given in \cite{Kas} and \cite{Matsumoto}, we prove Theorem \ref{thm1.2}. For this, we review the mapping class group of a closed surface, monodromy factorizations of Lefschetz fibrations, and Hurwitz equivalence for monodromy factorizations in order. Throughout this section, \(S_g\) denotes either \(\Sigma_g\) or \(N_g\). 

\subsection{Dehn twists, the mapping class group of a closed surface, and lifts.}\label{subsec.4.1}

Let \(c\) be a (two-sided) simple closed curve on \(S_g\) and \(\nu c\) a tubular neighborhood of \(c\).  A \textit{right handed Dehn twist} \(T_c\) about \(c\) is an auto-diffeomorphism of \(S_g\) which is drawn as Figure \ref{Fig.3}.

\begin{figure}[H]
\centering
\includegraphics[width=100mm]{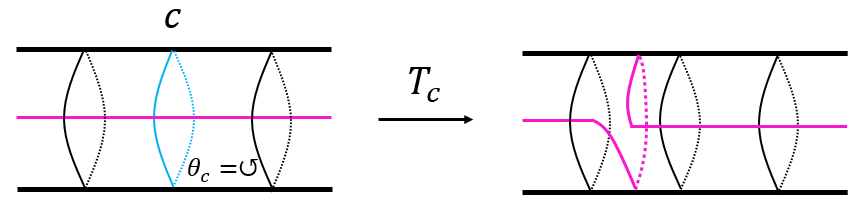}
\caption{A right handed Dehn twist \(T_c\) about \(c\).}
\label{Fig.3}
\end{figure}

\noindent
The isotopy class of \(T_c\) is also called the \textit{right handed Dehn twist} about \(c\), and we write this as \(t_c\). In the case of \(S_g=N_g\), we must choice an orientation of \(\nu c\) when doing a Dehn twist. So, for \(c \subset N_g\), let \(\theta_c\) denote an orientation of \(\nu c\), and then we represent the Dehn twist \(t_c\) as \(t_{c;\theta_c}\). However, note that the notation \(\theta_c\) may be omitted when an orientation of \(\nu c\) is given explicitly.
\par
Let \(\mathcal{M}(S_g)\) be the set of isotopy classes of auto-diffeomorphisms of \(S_g\), here if \(S_g=\Sigma_g\), we add a condition that auto-diffeomorphisms are orientation preserving. This forms a group and said to be the \textit{mapping class group} of \(S_g\). Note here that we give the operation on \(\mathcal{M}(S_g)\) by \(\lbrack \omega \rbrack \cdot \lbrack \omega' \rbrack := \lbrack \omega' \circ \omega \rbrack\) for \(\lbrack \omega \rbrack , \lbrack \omega' \rbrack \in \mathcal{M}(S_g)\). It is well-known that \(\mathcal{M}(\Sigma_g)\) is generated by Dehn twists (cf. \cite{Dehn}). Moreover, \(t_{c_1}t_{c_2}=t_{c_2}t_{c_1}\) if \(c_1\) and \(c_2\) are disjoint simple closed curves in \(\Sigma_{g}\) and \(t_{\omega(c)}=\Omega^{-1}t_c\Omega\) for \(\Omega=\lbrack \omega \rbrack \in \mathcal{M}(\Sigma_g)\) are also well-known. On the other hand, it is well-known fact that we need not only Dehn twists but also isotopy classes of other diffeomorphisms which are called \textit{crosscap slides} to generate \(\mathcal{M}(N_g)\) (cf. \cite{Lickorish1} and \cite{Lickorish2}). The following relations hold in \(\mathcal{M}(N_g)\):
\begin{equation}\label{(4.1)}
t_{c;\theta_c}=t_{c;-\theta_c}^{-1},
\end{equation}
and for \(\Omega=\lbrack \omega \rbrack \in \mathcal{M}(N_g)\),
\begin{equation}\label{(4.2)}
t_{\omega(c);\omega_{*}(\theta_c)}=\Omega^{-1}t_{c;\theta_c}\Omega,
\end{equation}
where \(-\theta_c\) is the opposite orientation with respect to \(\theta_c\) and \(\omega_{*}(\theta_c)\) means the orientation of \(\nu(\omega(c))\) induced by \(\theta_c\).
\par
Let us recall the diffeomorphism \(J:\Sigma_{g-1} \to \Sigma_{g-1}\) and the orientation double covering map \(\widetilde{J}:\Sigma_{g-1} \to N_g\) in Example \ref{examp2.3}. If a diffeomorphism \(\Psi:\Sigma_{g-1} \to \Sigma_{g-1}\) satisfies \(\Psi \circ J=J \circ \Psi\), then we can obtain a diffeomorphism \(\Psi':N_g \to N_g\) such that \(\widetilde{J} \circ \Psi=\Psi' \circ \widetilde{J}\). We call \(\Psi\) a \textit{lift} of \(\Psi'\).

\begin{lem}[\cite{B and C}, \cite{Szepietowski}]\label{lem4.1}
\textit{For \(g \geqq 3\), the map \(\eta:\mathcal{M}(N_g) \to \mathcal{M}(\Sigma_{g-1})\) defined by \(\eta(\lbrack \omega \rbrack):=\lbrack \widetilde{\omega} \rbrack\) is an injective homomorphism, where \(\widetilde{\omega}\) is exactly one orientation preserving lift of \(\omega\). 
}
\end{lem}

Fix a two-sided simple closed curve \(c\) on \(N_g\) and an orientation \(\theta_c\) of \(\nu c \subset N_g\). Since \(\nu c\) is an orientable submanifold of \(N_g\), it follows from Lemma \ref{lem2.4} (2) that \(\widetilde{J}|_{\widetilde{J}^{-1}(\nu c)}:\widetilde{J}^{-1}(\nu c) \to \nu c\) is an orientation double covering map and \(\widetilde{J}^{-1}(\nu c)\) is disconnected. In particular, \(\widetilde{J}^{-1}(\nu c)\) is the direct sum of tubular neighborhoods of two disjoint simple closed curves \(\widetilde{J}^{-1}(c)\), and their orientations are the opposite. So, let \(\gamma\) denote one of the direct sum components of \(\widetilde{J}^{-1}(c)\) such that the orientation of its tubular neighborhood coincides with \(\theta_c\), and we set \(\overline{\gamma}:=\widetilde{J}^{-1}(c)-\gamma\). It is well-known fact that \(\eta(t_{c;\theta_c})=t_{\gamma}t_{\overline{\gamma}}^{-1}\). In addition, we immediately have, by the relation \eqref{(4.2)}, that 
\begin{equation}\label{(4.3)}
\eta(t_{\omega(c);\omega_{*}(\theta_c)})=t_{\widetilde{\omega}(\gamma)}t_{\widetilde{\omega}(\overline{\gamma})}^{-1}
\end{equation}
for an auto-diffeomorphism \(\omega:N_g \to N_g\), where \(\widetilde{\omega}:\Sigma_{g-1} \to \Sigma_{g-1}\) is the orientation preserving lift of \(\omega\).

\subsection{Monodromy factorizations of Lefschetz fibrations.}

Let \(f:X \to \Sigma\) be a genus-\(g\) Lefschetz fibration. Fix \(y \in \Sigma-\Delta_f\) and a diffeomorphism \(\Xi:f^{-1}(y) \to S_g\), here \(\Xi\) is assumed to preserve orientation if \(S_g=\Sigma_g\). For the fiber bundle \(f|_{f^{-1}(\Sigma-\Delta_f)}:f^{-1}(\Sigma-\Delta_f) \to \Sigma-\Delta_f\), let \(a^{*}f|_{f^{-1}(\Sigma-\Delta_f)}:a^{*}f^{-1}(\Sigma-\Delta_f) \to \lbrack 0,1 \rbrack\) denote the pullback of \(f|_{f^{-1}(\Sigma-\Delta_f)}\) by a smooth loop \(a:\lbrack 0,1 \rbrack \to \Sigma-\Delta_f\). Let \(\Phi:\lbrack 0,1 \rbrack \times S_g \to a^{*}f^{-1}(\Sigma-\Delta_f)\) be a trivialization map of \(a^{*}f|_{f^{-1}(\Sigma-\Delta_f)}\) with \(\Phi(0,p)=(0,\Xi^{-1}(p))\) for all \(p \in S_g\), and for each \(t \in \lbrack 0,1 \rbrack\), define a map \(\Phi_t:S_g \to f^{-1}(y)\) as \(\Phi_t(p):=proj_2 \circ \Phi(t,p) \thickspace (p \in S_g)\), where \(proj_2:\lbrack 0,1 \rbrack \times f^{-1}(\Sigma-\Delta_f) \to f^{-1}(\Sigma-\Delta_f)\) is the second projection. Then we get a well-defined homomorphism \(\rho_f:\pi_1(\Sigma-\Delta_f,y) \to \mathcal{M}(S_g)\) which is defined by \(\rho_f(\lbrack a \rbrack):=\lbrack \Xi \circ \Phi_1 \rbrack\) for all \(\lbrack a \rbrack \in \pi_1(\Sigma-\Delta_f,y)\). This is called a \textit{monodromy representation} of \(f\).
\par
Let \(y_1,\cdots,y_n \in \Delta_f\), and we compose a smooth simple curve \(\mu_i:\lbrack 0,1 \rbrack \to \Sigma \thickspace (i=1,\cdots,n)\) with the following properties: 
\begin{itemize}
  \item for each \(i=1,\cdots,n\), \(\mu_i(0)=y\), \(\mu_i(1)=y_i\), and \(\mu_i(\lbrack 0,1)) \subset \Sigma-\Delta_f\);
  \item for \(i \ne j\), \(\mu_i(\lbrack 0,1 \rbrack) \cap \mu_j(\lbrack 0,1 \rbrack)=\{y\}\);
  \item when we travel counterclockwise on a sufficiently small circle centered at \(y\), the simple curves appear in the order \(\mu_1,\cdots,\mu_n\).
\end{itemize}      
The sequence of these simple curves with the above properties \((\mu_1,\cdots,\mu_n)\) is said to be a \textit{Hurwitz system} of \(f\). For each simple curve \(\mu_i\) of a Hurwitz system \((\mu_1,\cdots,\mu_n)\), we can take a smooth simple closed curve \(a_i:\lbrack 0,1 \rbrack \to \Sigma-\Delta_f\) which satisfies the following conditions:
\begin{itemize}
 \item \(a_i(0)=a_i(1)=y\);
 \item \(a_i\) goes counterclockwise around \(y_i\) and goes along with \(\mu_i\).
\end{itemize} 
For \(i=1, \cdots n\), we say that the homotopy class of \(a_i\) is the \textit{meridian loop} of \(f\) with respect to \(\mu_i\). It is well-known that \(\rho_f(\lbrack a_i \rbrack)\) is equal to the Dehn twist about the vanishing cycle of \(f\) determined by \(\mu_i\). Here, we note a few things about this. First, consider the case of \(f\) being orientable. In this case, the Dehn twist about the vanishing cycle determined by \(\mu_i\) is right handed (resp. left handed) if the critical point of \(f\) belonging to \(\mathrm{Crit}(f) \cap f^{-1}(y_i)\) is positive (resp. negative). Next, consider the case of \(f\) being non-orientable. From Remark \ref{rem3.2}, notice that the Dehn twist about each vanishing cycle of \(f\) can be defined. In this case, whether the Dehn twist about the vanishing cycle determined by \(\mu_i\) is right handed or left handed depends on an orientation of a tubular neighborhood of the vanishing cycle because of the relation \eqref{(4.1)}.
\par
Now, let \(f:X \to \Sigma\) be a genus-\(g\) non-orientable Lefschetz fibration, and we use exactly the same symbols for \(f\) as we used above. Notice that the codomain of a fixed diffeomorphism \(\Xi\) is equal to \(N_g\). In what follows, \(\pi:\widetilde{X} \to X\) is the standard orientation double covering map.

\begin{prop}\label{prop4.2}
\textit{
Let \(\eta:\mathcal{M}(N_g) \to \mathcal{M}(\Sigma_{g-1})\) be the injective homomorphism defined in Lemma \ref{lem4.1}. For \(g \geqq 3\), the following equality holds}:
\[
\rho_{f \circ \pi}=\eta \circ \rho_f.
\]
\end{prop}

\begin{proof}
It is easily seen that the domains of \(\rho_{f \circ \pi}\) and \(\eta \circ \rho_f\) are the same since \(\Sigma-\Delta_{f \circ \pi}=\Sigma-\Delta_f\) holds. By applying Lemma \ref{lem2.5}, for a fixed diffeomorphism \(\Xi:f^{-1}(y) \to N_g\), we get a unique orientation preserving diffeomorphism \(\widetilde{\Xi}:(f \circ \pi)^{-1}(y) \to \Sigma_{g-1}\) such that the diagram (\ref{diagram(2.1)}) commutes. Take the pull-back \(a^{*}(f \circ \pi)|_{(f \circ \pi)^{-1}(\Sigma-\Delta_f)}:a^{*}(f \circ \pi)^{-1}(\Sigma-\Delta_f) \to \lbrack 0,1 \rbrack\) of the fiber bundle \((f \circ \pi)|_{(f \circ \pi)^{-1}(\Sigma-\Delta_f)}:(f \circ \pi)^{-1}(\Sigma-\Delta_f) \to \Sigma-\Delta_f\) by \(a\). Note that \((id_{\lbrack 0,1 \rbrack} \times \pi)|_{a^{*}(f \circ \pi)^{-1}(\Sigma-\Delta_f)}:a^{*}(f \circ \pi)^{-1}(\Sigma-\Delta_f) \to a^{*}f^{-1}(\Sigma-\Delta_f)\) is an orientation double covering map since \(a^{*}f^{-1}(\Sigma-\Delta_f)\) is a submanifold of \(\lbrack 0,1 \rbrack \times f^{-1}(\Sigma-\Delta_f)\) and \(a^{*}(f \circ \pi)^{-1}(\Sigma-\Delta_f)=\pi^{-1}(a^{*}f^{-1}(\Sigma-\Delta_f))\). Lemma \ref{lem2.5} hence gives a unique orientation preserving diffeomorphism \(\widetilde{\Phi}:\lbrack 0,1 \rbrack \times \Sigma_{g-1} \to a^{*}(f \circ \pi)^{-1}(\Sigma-\Delta_f)\) for the diffeomorphism \(\Phi:\lbrack 0,1 \rbrack \times N_g \to a^{*}f^{-1}(\Sigma-\Delta_f)\) so that the upper rectangle of the diagram
\begin{equation}\label{diagram(4.4)}
\xymatrix{
{\lbrack 0,1 \rbrack \times \Sigma_{g-1}}\ar[rr]^-{\widetilde{\Phi}}\ar[d]_-{id_{\lbrack 0,1 \rbrack} \times \widetilde{J}}\ar@{}[rd]|{ }&&{a^{*}(f \circ \pi)^{-1}(\Sigma-\Delta_f)}\ar[d]^-{(id_{\lbrack 0,1 \rbrack} \times \pi)|_{a^{*}(f \circ \pi)^{-1}(\Sigma-\Delta_f)}}\\
{\lbrack 0,1 \rbrack \times N_g}\ar[rr]^-{\Phi}\ar[dr]_-{proj_1}&\ar@{}[rd]|{ }&{a^{*}f^{-1}(\Sigma-\Delta_f)}\ar[dl]^-{\thickspace \thickspace \thickspace \thickspace a^{*}f|_{f^{-1}(\Sigma-\Delta_f)}}\\
&\lbrack 0,1 \rbrack&
}
\end{equation}
commutes, where \(id_{\lbrack 0,1 \rbrack}:\lbrack 0,1 \rbrack \to \lbrack 0,1 \rbrack\) is the identity map and \(proj_1:\lbrack 0,1 \rbrack \times N_g \to \lbrack 0,1 \rbrack\) is the first projection. Since \(\Phi\) is a trivialization map, the lower triangle of the diagram (\ref{diagram(4.4)}) is commutative. So, it follows that the diagram (\ref{diagram(4.4)}) commutes, which implies that \(\widetilde{\Phi}\) is a trivialization map of \(a^{*}(f \circ \pi)|_{(f \circ \pi)^{-1}(\Sigma-\Delta_f)}\). From the commutativity of the diagram (\ref{diagram(4.4)}), we have \(\widetilde{\Phi}(0,p) \in \{0\} \times \pi^{-1}(\pi(\widetilde{\Xi}^{-1}(p)))\) for all \(p \in \Sigma_{g-1}\). Define a map \(r:(f \circ \pi)^{-1}(y) \to (f \circ \pi)^{-1}(y)\) to be \(r(x,+o_x):=(x,-o_x)\) and \(r(x,-o_x):=(x,+o_x)\) for any \((x,+o_x),(x,-o_x) \in (f \circ \pi)^{-1}(y)\), which is the orientation reversing involution belonging to \(\mathrm{Aut}(\pi|_{(f \circ \pi)^{-1}(y)})\). Then for each \(p \in \Sigma_{g-1}\), the set \(\pi^{-1}(\pi(\widetilde{\Xi}^{-1}(p)))\) is equal to \(\{\widetilde{\Xi}^{-1}(p),(r \circ \widetilde{\Xi}^{-1})(p)\}\). Noticing that \(\widetilde{\Phi}\) is orientation preserving, we observe that \(\widetilde{\Phi}(0,p)\) equals to \((0,\widetilde{\Xi}^{-1}(p))\) for all \(p \in \Sigma_{g-1}\). Therefore, a monodromy representation \(\rho_{f \circ \pi}:\pi_1(\Sigma-\Delta_{f \circ \pi},y) \to \mathcal{M}(\Sigma_{g-1})\) of \(f \circ \pi\) can be defined as \(\rho_{f \circ \pi}(\lbrack a \rbrack):=\lbrack \widetilde{\Xi} \circ \widetilde{\Phi}_1 \rbrack\) for all \(\lbrack a \rbrack \in \pi_1(\Sigma-\Delta_{f \circ \pi},y)\), where the map \(\widetilde{\Phi}_1:\Sigma_{g-1} \to (f \circ \pi)^{-1}(y)\) is defined by \(\widetilde{\Phi}_1(p):=(\widetilde{proj_2} \circ \widetilde{\Phi})(1,p)\) for all \(p \in \Sigma_{g-1}\) and \(\widetilde{proj_2}:\lbrack 0,1 \rbrack \times (f \circ \pi)^{-1}(\Sigma - \Delta_{f \circ \pi}) \to (f \circ \pi)^{-1}(\Sigma - \Delta_{f \circ \pi})\) is the second projection. Moreover, it immediately follows from the commutative diagram (\ref{diagram(4.4)}) that the following diagram is commutative:
\begin{equation}\label{diagram(4.5)}
\xymatrix{
\Sigma_{g-1}\ar[r]^-{\widetilde{\Phi}_1}\ar[d]_-{\widetilde{J}}\ar@{}[rd]|{ }
&(f \circ \pi)^{-1}(y)\ar[r]^-{\widetilde{\Xi}}\ar[d]_-{\pi|_{(f \circ \pi)^{-1}(y)}}\ar@{}[rd]|{ }&\Sigma_{g-1}\ar[d]^-{\widetilde{J}}\\
N_g \ar[r]_-{\Phi_1}&f^{-1}(y) \ar[r]_-{\Xi}&N_g.
}
\end{equation}
The above comuutative diagram (\ref{diagram(4.5)}) shows that \(\widetilde{\Xi} \circ \widetilde{\Phi}_1\) is a lift of \(\Xi \circ \Phi_1\), that is, \(\eta(\lbrack \Xi \circ \Phi_1 \rbrack)=\lbrack \widetilde{\Xi} \circ \widetilde{\Phi}_1 \rbrack\) holds. This yields the assertion.
\end{proof}

\par
The rest of this subsection, suppose that \(\Sigma\) equals to either \(\mathbb{D}^2\) or \(\mathbb{S}^2\). Let \(f:X \to \Sigma\) be a Lefschetz fibration of genus-\(g\), and fix \(y \in \Sigma-\Delta_f\) and a diffeomorphism \(\Xi:f^{-1}(y) \to S_g\), here we asuume that \(\Xi\) is orirentation preserving if \(f\) is orientable. Also, let \(\zeta_1, \cdots ,\zeta_n \in \{1,-1\}\). Choosing a Hurwitz system \((\mu_1,\cdots,\mu_n)\) of \(f\), we obtain the vanishing cycle \(C_i \subset S_g\) of \(f\) determined by \(\mu_i\) and the meridian loop \(\alpha_i \in \pi_1(\Sigma-\Delta_f,y)\) of \(f\) with respect to \(\mu_i\) (\(i=1, \cdots , n\)).  We first consider the case of \(\Sigma=\mathbb{D}^2\). It is obvious that the product \(\alpha_1 \cdots \alpha_n\) is equal to the homotopy class \(\lbrack \partial \mathbb{D}^2 \rbrack\) in \(\pi_1(\mathbb{D}^2-\Delta_f,y)\). Therefore, mapping both sides of this equation by a monodromy representation \(\rho_f:\pi_1(\mathbb{D}^2-\Delta_f,y) \to \mathcal{M}(S_g)\) of \(f\), the monodromy of \(\lbrack \partial \mathbb{D}^2 \rbrack\) is represented as \(t_{C_1}^{\zeta_1}, \cdots ,t_{C_n}^{\zeta_n}\). This product of Dehn twists is called a \textit{monodromy factorization} of \(f\). Next, consider the case of \(\Sigma=\mathbb{S}^2\). Since \(\pi_1(\mathbb{S}^2-\Delta_f,y)=\langle \alpha_1, \cdots ,\alpha_n \thickspace | \thickspace \alpha_1 \cdots \alpha_n=1 \rangle\), the relation \(\alpha_1 \cdots \alpha_n=1\) is provided. We then get the relation \(t_{C_1}^{\zeta_1}, \cdots ,t_{C_n}^{\zeta_n}=1\) in \(\mathcal{M}(S_g)\) by \(\rho_f\), which is also said to be a \textit{monodromy factorization} of \(f\). 
\par
According to the above, we see that a relation about a product of Dehn twists is obtained from a Lefschetz fibration. On the other hand, the converse also holds, that is, if (two-sided) simple closed curves \(c_1, \cdots ,c_n\) in \(S_g\) are given, then there is a gunus-\(g\) Lefschetz fibration over \(\mathbb{D}^2\) and its Hurwitz system \((\mu_1,\cdots,\mu_n)\) such that each \(c_i\) is the vanishing cycle of \(f\) determined by \(\mu_i\) (\(i=1, \cdots ,n)\). In a similar fashion, if (two-sided) simple closed curves \(c_1, \cdots ,c_n\) in \(S_g\) satisfying \(t_{c_1}^{\zeta_1}, \cdots ,t_{c_n}^{\zeta_n}=1\) in \(\mathcal{M}(S_g)\) are given, then there is a gunus-\(g\) Lefschetz fibration over \(\mathbb{S}^2\) and its Hurwitz system \((\mu_1,\cdots,\mu_n)\) equipped with the same condition above.
\par
Using the relation \eqref{(4.1)}, a monodromy factorization of a non-orientable Lefschetz fibration can be consturucted by a product of right handed Dehn twists. We call such monodromy factorization \textit{positive}.
\par
The next corollary follows directly from Proposition \ref{prop4.2}.

\begin{cor}\label{cor4.3}
\textit{Assume that \(g \geqq 3\) and a positive monodromy factorization \(t_{{c}_1;\theta_{{c}_1}} \cdots t_{{c}_n;\theta_{{c}_n}}\) of a non-orientable Lefschetz fibration \(f\) is given. Let \(\gamma_i \amalg \overline{\gamma_i}\) be the inverse image of \(c_i\) under \(\widetilde{J}\) determined by \(\theta_{c_i}\) as the last part of Subsection \ref{subsec.4.1} \((i=1, \cdots ,n)\). A monodromy factorization of the achiral Lefschetz fibration \(f \circ \pi\) is given by \(t_{{\gamma}_1}t_{\overline{{\gamma}_1}}^{-1} \cdots t_{{\gamma}_n}t_{\overline{{\gamma}_n}}^{-1}\).}
\end{cor}

\begin{proof}
We suppose that a Hurwitz system \((\mu_1,\cdots,\mu_n)\) of \(f\) and the meridian loops \(\alpha_1, \cdots ,\alpha_n\) about it give the mondromy factorization  \(t_{{c}_1;\theta_{{c}_1}} \cdots t_{{c}_n;\theta_{{c}_n}}\). Since \(\Delta_{f \circ \pi}=\Delta_{f}\), we can take \((\mu_1,\cdots,\mu_n)\) as a Hurwitz system of \(f\circ \pi\). In particular, it follows that \(\alpha_1, \cdots ,\alpha_n\) are meridian loops of \(f \circ \pi\). By applying Proposition \ref{prop4.2},
\begin{equation*}
    \begin{split}
\rho_{f \circ \pi}(\alpha_1 \cdots \alpha_n)
&=(\eta \circ \rho_{f})(\alpha_1 \cdots \alpha_n) \\
&=\eta(t_{{c}_1;\theta_{{c}_1}} \cdots t_{{c}_n;\theta_{{c}_n}}) \\
&=t_{{\gamma}_1}t_{\overline{{\gamma}_1}}^{-1} \cdots t_{{\gamma}_n}t_{\overline{{\gamma}_n}}^{-1}.
    \end{split}
\end{equation*}
This leads to the desired conclusion.
\end{proof}

\subsection{Isomorphism classes of Lefschetz fibrations and Hurwitz equivalence.}

Throughout this subsection, we assume that Lefschetz fibrations are \textit{relatively minimal}, that is, no singular fiber contains a sphere with self intersection number \(-1\). 

\begin{dfn}
Let \(f_1:X_1 \to \Sigma\) and \(f_2:X_2 \to \Sigma\) be Lefschetz fibrations. They are said to be \textit{isomorphic} if there exist a diffeomorphism \(H:X_1 \to X_2\) and an orientation preserving diffeomorphism \(h:\Sigma \to \Sigma\) such that \(f_2 \circ H=h \circ f_1\), here \(H\) is assumed to be orientation preserving if \(f_1\) and \(f_2\) are orientable.
\end{dfn}

We now introduce the following transformations (HE1-1)-(HE2) between two products of Dehn twists in \(\mathcal{M}(\Sigma_{g})\): For \(\zeta_i \in \{1,-1\} \thickspace (i=1, \cdots ,n)\),
\begin{itemize}
\item[(HE1)]\underline{\textit{elementary transformations}}
\begin{itemize}
\item[(HE1-1)] 
\begin{gather*}
{t_{c_1}}^{\zeta_1} \cdots {t_{c_i}}^{\zeta_i} {t_{c_{i+1}}}^{\zeta_{i+1}} {t_{c_{i+2}}}^{\zeta_{i+2}}       \cdots {t_{c_n}}^{\zeta_n} \\
\leftrightarrow {t_{c_1}}^{\zeta_1} \cdots {t_{c_i}}^{\zeta_i} {t_{c_{i+1}}}^{\zeta_{i+1}} ({t_{c_i}}^{- \zeta_i}{t_{c_i}}^{\zeta_i}) {t_{c_{i+2}}}^{\zeta_{i+2}} \cdots {t_{c_n}}^{\zeta_n};
\end{gather*}
\item[(HE1-2)]
\begin{gather*}
{t_{c_1}}^{\zeta_1} \cdots {t_{c_{i-1}}}^{\zeta_{i-1}}  {t_{c_i}}^{\zeta_i} {t_{c_{i+1}}}^{\zeta_{i+1}}  \cdots {t_{c_n}}^{\zeta_n} \\
\leftrightarrow {t_{c_1}}^{\zeta_1} \cdots {t_{c_{i-1}}}^{\zeta_{i-1}} ({t_{c_{i+1}}}^{\zeta_{i+1}} {t_{c_{i+1}}}^{-\zeta_{i+1}}) {t_{c_i}}^{\zeta_i} {t_{c_{i+1}}}^{\zeta_{i+1}} \cdots {t_{c_n}}^{\zeta_n};
\end{gather*}
\end{itemize}
\item[(HE2)]\underline{\textit{simultaneous conjugations}}
\begin{equation*}
{t_{c_1}}^{\zeta_1} \cdots {t_{c_n}}^{\zeta_n} \leftrightarrow {t_{\sigma(c_1)}}^{\zeta_1} \cdots {t_{\sigma(c_n)}}^{\zeta_n}
\end{equation*}
for an orientation preserving auto-diffeomorphism \(\sigma:\Sigma_{g} \to \Sigma_{g}\).
\end{itemize}

\noindent
If two products of Dehn twists change each other by the above transformations, then they are called \textit{Hurwitz equivalent}. The next theorem states that (HE1-1)-(HE2) are very informative for conducting combinatorial studies about orientable Lefschetz fibrations, i.e., orientable 4-manifolds.

\begin{thm}[\cite{Kas}, \cite{Matsumoto}]\label{thm4.5}
\textit{
Let \(f_i:X_i \to \Sigma\) be a genus-\(g\) orientable Lefschetz fibration over either \(\mathbb{D}^2\) or \(\mathbb{S}^2\) \((i=1,2)\), where if \(\Sigma=\mathbb{S}^2\), suppose that \(g \geqq 2\). Then \(f_1\) is isomorphic to \(f_2\) if and only if their monodromy factorizations are Hurwitz equivalent.
}  
\end{thm}

We remark that it follows from elementary transformations that the isomorphism class of an orientable Lefschetz fibration does not change if the following substitution is applied to its monodromy factorization: 
\begin{gather*}
{t_{c_1}}^{\zeta_1} \cdots {t_{c_{i-1}}}^{\zeta_{i-1}} \underline{{t_{c_i}}^{\zeta_i} {t_{c_{i+1}}}^{\zeta_{i+1}}} {t_{c_{i+2}}}^{\zeta_{i+2}} \cdots {t_{c_n}}^{\zeta_n} \\
\leftrightarrow {t_{c_1}}^{\zeta_1} \cdots {t_{c_{i-1}}}^{\zeta_{i-1}} \underline{{t_{c_{i+1}}}^{\zeta_{i+1}} {t_{c_i}}^{\zeta_i}} {t_{c_{i+2}}}^{\zeta_{i+2}} \cdots {t_{c_n}}^{\zeta_n}
\end{gather*}
for \(\zeta_i \in \{1,-1\} \thickspace (i=1, \cdots ,n)\) and disjoint simple closed curves \(c_i\) and \(c_{i+1}\) in \(\Sigma_{g}\).
\par
Let us prove the main theorem. From now on, suppose that \(g \geqq 3\) and \(\Sigma\) is equal to either \(\mathbb{D}^2\) or \(\mathbb{S}^2\). 

\begin{proof}[Proof of Theorem \ref{thm1.2}.]
The `only if' part is proved by the same way as in the orientable case (see \cite{Kas} and \cite{Matsumoto}).\par
We show the `if' part. Let \(f:X \to \Sigma\) be a genus-\(g\) non-orientable Lefschetz fibration with a positive monodromy factorization \(t_{{c}_1;\theta_{{c}_1}} \cdots t_{{c}_n;\theta_{{c}_n}}\) and \(\pi:\widetilde{X} \to X\) the standard orientation double covering map. The product \({t_{c_1;\theta_{c_1}}} \cdots {t_{c_i;\theta_{c_i}}} {t_{c_{i+1};\theta_{c_{i+1}}}}({t_{c_i;\theta_{c_i}}}^{-1} {t_{c_i;\theta_{c_i}}}) {t_{c_{i+2};\theta_{c_{i+2}}}} \cdots {t_{c_n;\theta_{c_{n}}}}\) gives a genus-\(g\) non-orientable Lefschetz fibration, and let \(f_{\mathrm{i-1}}:X_{\mathrm{i-1}} \to \Sigma\) denote this. According to Corollary \ref{cor4.3}, we see that monodromy factorizations of \(f \circ \pi\) and \(f_{\mathrm{i-1}} \circ \pi_{\mathrm{i-1}}\) are given by \(t_{{\gamma}_1}t_{\overline{{\gamma}_1}}^{-1} \cdots t_{{\gamma}_n}t_{\overline{{\gamma}_n}}^{-1}\) and \(t_{{\gamma}_1}t_{\overline{{\gamma}_1}}^{-1} \cdots t_{{\gamma}_i}t_{\overline{{\gamma}_i}}^{-1} t_{{\gamma}_{i+1}}t_{\overline{{\gamma}_{i+1}}}^{-1} (t_{\overline{{\gamma}_i}}t_{{\gamma}_i}^{-1} \\ t_{{\gamma}_i}  t_{\overline{{\gamma}_i}}^{-1}) t_{{\gamma}_{i+2}}t_{\overline{{\gamma}_{i+2}}}^{-1} \cdots t_{{\gamma}_n}t_{\overline{{\gamma}_n}}^{-1}\), respectively, where \(\pi_{\mathrm{i-1}}:\widetilde{X_{\mathrm{i-1}}} \to X_{\mathrm{i-1}}\) indicates the standard orientation double covering map. Then
\begin{equation*}
\begin{split} 
&t_{{\gamma}_1}t_{\overline{{\gamma}_1}}^{-1} \cdots \underline{t_{{\gamma}_i}t_{\overline{{\gamma}_i}}^{-1}} t_{{\gamma}_{i+1}}t_{\overline{{\gamma}_{i+1}}}^{-1}  t_{{\gamma}_{i+2}}t_{\overline{{\gamma}_{i+2}}}^{-1} \cdots t_{{\gamma}_n}t_{\overline{{\gamma}_n}}^{-1} \\
&=t_{{\gamma}_1}t_{\overline{{\gamma}_1}}^{-1} \cdots \underline{t_{\overline{{\gamma}_i}}^{-1}t_{{\gamma}_i}} t_{{\gamma}_{i+1}}t_{\overline{{\gamma}_{i+1}}}^{-1}  t_{{\gamma}_{i+2}}t_{\overline{{\gamma}_{i+2}}}^{-1} \cdots t_{{\gamma}_n}t_{\overline{{\gamma}_n}}^{-1} \\
&\xleftrightarrow{\mathrm{(HE1-1)}} t_{{\gamma}_1}t_{\overline{{\gamma}_1}}^{-1} \cdots t_{\overline{{\gamma}_i}}^{-1}t_{{\gamma}_i} t_{{\gamma}_{i+1}} (t_{{\gamma}_i}^{-1}t_{{{\gamma}_i}}) t_{\overline{{\gamma}_{i+1}}}^{-1}  t_{{\gamma}_{i+2}}t_{\overline{{\gamma}_{i+2}}}^{-1} \cdots t_{{\gamma}_n}t_{\overline{{\gamma}_n}}^{-1} \\
&\xleftrightarrow{\mathrm{(HE1-1)}} t_{{\gamma}_1}t_{\overline{{\gamma}_1}}^{-1} \cdots t_{\overline{{\gamma}_i}}^{-1}t_{{\gamma}_i} t_{{\gamma}_{i+1}} (t_{{\gamma}_i}^{-1}t_{{{\gamma}_i}}) t_{\overline{{\gamma}_{i+1}}}^{-1} (t_{{\gamma}_i}^{-1}t_{{{\gamma}_i}})  t_{{\gamma}_{i+2}}t_{\overline{{\gamma}_{i+2}}}^{-1} \cdots t_{{\gamma}_n}t_{\overline{{\gamma}_n}}^{-1} \\
&\xleftrightarrow{\mathrm{(HE1-2)}} t_{{\gamma}_1}t_{\overline{{\gamma}_1}}^{-1} \cdots \underline{t_{\overline{{\gamma}_i}}^{-1}t_{{\gamma}_i}} t_{{\gamma}_{i+1}} t_{\overline{{\gamma}_{i+1}}}^{-1} (t_{{\gamma}_i}^{-1}t_{{{\gamma}_i}})  t_{{\gamma}_{i+2}}t_{\overline{{\gamma}_{i+2}}}^{-1} \cdots t_{{\gamma}_n}t_{\overline{{\gamma}_n}}^{-1} \\
&= t_{{\gamma}_1}t_{\overline{{\gamma}_1}}^{-1} \cdots \underline{t_{{\gamma}_i}t_{\overline{{\gamma}_i}}^{-1}} t_{{\gamma}_{i+1}} t_{\overline{{\gamma}_{i+1}}}^{-1} (t_{{\gamma}_i}^{-1}t_{{{\gamma}_i}})  t_{{\gamma}_{i+2}}t_{\overline{{\gamma}_{i+2}}}^{-1} \cdots t_{{\gamma}_n}t_{\overline{{\gamma}_n}}^{-1} \\
&\xleftrightarrow{\mathrm{(HE1-1)}} t_{{\gamma}_1}t_{\overline{{\gamma}_1}}^{-1} \cdots t_{{\gamma}_i}t_{\overline{{\gamma}_i}}^{-1} t_{{\gamma}_{i+1}} (t_{\overline{{\gamma}_i}}t_{\overline{{\gamma}_i}}^{-1}) t_{\overline{{\gamma}_{i+1}}}^{-1} (t_{{\gamma}_i}^{-1}t_{{{\gamma}_i}})  t_{{\gamma}_{i+2}}t_{\overline{{\gamma}_{i+2}}}^{-1} \cdots t_{{\gamma}_n}t_{\overline{{\gamma}_n}}^{-1} \\
&\xleftrightarrow{\mathrm{(HE1-1)}} t_{{\gamma}_1}t_{\overline{{\gamma}_1}}^{-1} \cdots t_{{\gamma}_i}t_{\overline{{\gamma}_i}}^{-1} t_{{\gamma}_{i+1}} (t_{\overline{{\gamma}_i}}t_{\overline{{\gamma}_i}}^{-1}) t_{\overline{{\gamma}_{i+1}}}^{-1} (t_{\overline{{\gamma}_i}}t_{\overline{{\gamma}_i}}^{-1}) (t_{{\gamma}_i}^{-1}t_{{{\gamma}_i}})  t_{{\gamma}_{i+2}}t_{\overline{{\gamma}_{i+2}}}^{-1} \cdots t_{{\gamma}_n}t_{\overline{{\gamma}_n}}^{-1} \\
&\xleftrightarrow{\mathrm{(HE1-2)}} t_{{\gamma}_1}t_{\overline{{\gamma}_1}}^{-1} \cdots t_{{\gamma}_i}t_{\overline{{\gamma}_i}}^{-1} t_{{\gamma}_{i+1}} t_{\overline{{\gamma}_{i+1}}}^{-1} (t_{\overline{{\gamma}_i}}\underline{t_{\overline{{\gamma}_i}}^{-1}) (t_{{\gamma}_i}^{-1}t_{{{\gamma}_i}}})  t_{{\gamma}_{i+2}}t_{\overline{{\gamma}_{i+2}}}^{-1} \cdots t_{{\gamma}_n}t_{\overline{{\gamma}_n}}^{-1} \\
&=t_{{\gamma}_1}t_{\overline{{\gamma}_1}}^{-1} \cdots t_{{\gamma}_i}t_{\overline{{\gamma}_i}}^{-1} t_{{\gamma}_{i+1}}t_{\overline{{\gamma}_{i+1}}}^{-1} (t_{\overline{{\gamma}_i}}\underline{t_{{\gamma}_i}^{-1} t_{{\gamma}_i}t_{\overline{{\gamma}_i}}^{-1}}) t_{{\gamma}_{i+2}}t_{\overline{{\gamma}_{i+2}}}^{-1} \cdots t_{{\gamma}_n}t_{\overline{{\gamma}_n}}^{-1},
\end{split}
\end{equation*}
which shows that \(f \circ \pi\) is isomorphic to \(f_{\mathrm{i-1}} \circ \pi_{\mathrm{i-1}}\) from Theorem \ref{thm4.5}. By the same augment, we also have that \(f \circ \pi\) is isomorphic to \(f_{\mathrm{i-2}} \circ \pi_{\mathrm{i-2}}\), where \(f_{\mathrm{i-2}}:X_{\mathrm{i-2}} \to \Sigma\) is a genus-\(g\) non-orientable Lefschetz fibration which is given by \({t_{c_1;\theta_{c_1}}} \cdots {t_{c_{i-1};\theta_{c_{i-1}}}} ({t_{c_{i+1};\theta_{c_{i+1}}}} {t_{c_{i+1};\theta_{c_{i+1}}}}^{-1}) {t_{c_{i};\theta_{c_{i}}}} {t_{c_{i+1};\theta_{c_{i+1}}}} \cdots {t_{c_n;\theta_{c_{n}}}}\) and \(\pi_{\mathrm{i-2}}:\widetilde{X_{\mathrm{i-2}}} \to X_{\mathrm{i-2}}\) denotes the standard orientation double covering map. Moreover, let \(f_{\mathrm{ii}}:X_{\mathrm{ii}} \to \Sigma\) be a genus-\(g\) non-orientable Lefschetz fibration which is constructed by \({t_{\omega(c_1);\omega_*(\theta_{c_1})}} \cdots {t_{\omega(c_n);\omega_*(\theta_{c_n})}}\)  for a fixed auto-diffeomorphism \(\omega:N_g \to N_g\). It follows from the same discussion in the proof of Corollary \ref{cor4.3} and the relation \eqref{(4.3)} that a monodromy factorization of \(f_{\mathrm{ii}} \circ \pi_{\mathrm{ii}}\) is produced by \(t_{\widetilde{\omega}(\gamma_1)}t_{\widetilde{\omega}(\overline{\gamma_1})}^{-1} \cdots t_{\widetilde{\omega}(\gamma_n)}t_{\widetilde{\omega}(\overline{\gamma_n})}^{-1}\), here \(\pi_{\mathrm{ii}}:\widetilde{X_{\mathrm{ii}}} \to X_{\mathrm{ii}}\) is the standard orientation double covering map and  \(\widetilde{\omega}:\Sigma_{g-1} \to \Sigma_{g-1}\) is the orientation preserving lift of \(\omega\). Theorem \ref{thm4.5} thus states that \(f \circ \pi\) is isomorphic to \(f_{\mathrm{ii}} \circ \pi_{\mathrm{ii}}\). More generally, we can find that if positive monodromy factorizations of two genus-\(g\) non-orientable Lefschetz fibrations \(f_j:X_j \to \Sigma\) change each other by applying the transformations (i-1)-(ii) in Theorem \ref{thm1.2} a finite number of times, then the genus-\((g-1)\) achiral Lefschetz fibtations \(f_j \circ \pi_j:\widetilde{X_j} \to \Sigma\) determined by the standard orientation double covering maps \(\pi_j:\widetilde{X_j} \to X_j\) are isomorphic \((j=1,2)\). Hence, there are orientation preserving diffeomorphisms \(\widetilde{H}:\widetilde{X_1} \to \widetilde{X_2}\) and \(h:\Sigma \to \Sigma\) such that \((f_2 \circ \pi_2) \circ \widetilde{H}=h \circ (f_1 \circ \pi_1)\), i.e., we get the following commutative diagram:
\begin{equation*}%\label{diagram}
\xymatrix{
\widetilde{X_1}\ar[r]^-{\widetilde{H}}\ar[d]_-{\pi_1}\ar@{}[rd]|{ }&\widetilde{X_2}\ar[d]^-{\pi_2}\\
X_1 \ar[d]_-{f_1} &X_2\ar[d]^-{f_2}\\
\Sigma \ar[r]_-{h}&\Sigma.
}
\end{equation*}
Since \(\pi_2 \circ \widetilde{H}\) is an orientation double covering map, for the orientation reversing involution \(\overline{r} \in \mathrm{Aut}(\pi_2 \circ \widetilde{H})\), \((\pi_2 \circ \widetilde{H}) \circ \overline{r}=\pi_2 \circ \widetilde{H}\) holds. In particular, \((\pi_2 \circ \widetilde{H})(x,+o_x)=(\pi_2 \circ \widetilde{H})(x,-o_x)\) for each \(x \in X_1\). Therefore, we can define a map \(H:X_1 \to X_2\) as \(H(x):=(\pi_2 \circ \widetilde{H})(x,o_x)\) for any \(x \in X_1\). Since \(\pi_1\) and \(\pi_2\) are surjective submersions, it follows that \(H\) is a diffeomorphism. Furthermore, we immediately see that \(H\) satisfies \(f_2 \circ H=h \circ f_1\), and thus, \(f_1\) is isomorphic to \(f_2\).
\par
The proof of Theorem \ref{thm1.2} is complete.
\end{proof}

Department of Mathematics, Faculty of Science, Okayama University, Okayama 700-8530, Japan \par
\textit{Email address}: p4bd8ad0@s.okayama-u.ac.jp

\end{document}